\documentclass[11pt]{article}
\usepackage{amsmath, amsthm, amssymb, mathrsfs, amscd, amsthm, amscd,amsfonts,graphicx}
\usepackage{indentfirst,url,xypic}
\usepackage{lipsum} 
\usepackage[all]{xy}
\usepackage{url}
\usepackage[colorlinks,plainpages]{hyperref}
\usepackage{verbatim}
\usepackage{fancyvrb}
\usepackage{listings}

\usepackage[colorlinks,plainpages]{hyperref}
\hypersetup{
	colorlinks=true,
	linkcolor=blue,
	filecolor=magenta,      
	urlcolor=cyan,
}

\DeclareMathOperator{\characteristic}{char}
\DeclareMathOperator{\order}{ord}

\newtheorem{Definition}{ Definition}[section]
\newtheorem{theorem}[Definition]{Theorem}
\newtheorem{proposition}[Definition]{Proposition }

\newcommand{\N}{\mathbb{N}}

\newcommand{\A}{\mathbb{A}}

\begin{document}
	\title{\bf{A note on finite determinacy of matrices}}
	\author{
		\bf{Thuy Huong Pham and Pedro Macias Marques}\\
	}
	\date{}
\maketitle
{\centering{\footnotesize\textit{ Dedicated to Professor Gert-Martin Greuel
	on the occasion of his seventy-fifth birthday}}\par}
\vskip 25pt
\begin{abstract}
	In this note, we give a necessary and sufficient condition for a matrix $A\in M_{2,2}$ to be  finitely $G$-determined, where $M_{2,2}$ is the ring of $2 \times 2$ matrices whose entries are formal power series over an infinite field, and  $G$ is a group acting on $M_{2,2}$ by change of coordinates together with multiplication by invertible matrices from both sides.
\end{abstract}



\section{Introduction}\label{introduction}
Throughout this paper let $K$ be an infinite field of arbitrary characteristic and 
\[R:=K[[{\bf{x}}]]=K[[x_1,\ldots, x_s]]\] 
the formal power series ring over $K$ in $s$ variables with maximal ideal $\mathfrak{m} = \langle x_1, \ldots , x_s \rangle$. We denote by 
\[ {M}_{m,n}:=Mat(m,n, R)\]
the set of all $m\times n$ matrices with entries in $R$. 
Let $G$ denote the group
\[
G:=\left(GL(m,R)\times GL(n,R)^\text{op}\right)\rtimes Aut(R),
\]
where $GL(n,R)^\text{op}$ is the opposite group of the group $GL(n,R)$ and $Aut(R)$ is the group of automorphisms defined on $R$. The group 
$G$ acts on  $M_{m,n}$ as follows
\[
(U,V,\phi, A)\mapsto U\cdot\phi (A)\cdot V,
\]
where $A=[a_{ij}({\bf{x}})]\in M_{m,n}$, $U\in GL(m,R)$, $V\in GL(n,R)^\text{op}$,  and $\phi (A):=[\phi(a_{ij}({\bf{x}}))]=[a_{ij}(\phi ({\bf{x}}))]$ with $\phi ({\bf{x}}):=(\phi_1,\ldots,\phi_s)$, $\phi_i:=\phi(x_i)\in\mathfrak{m}$ for all $i=1,\ldots,s$. 
Two matrices $A, B \in   M_{m,n}$ are called  {\it $G$--equivalent},  denoted  $A\mathop\sim\limits^{G} B$,  if $B$ lies in the orbit of A. We say that $A\in M_{m,n}$ is \emph{$G$ $k$--determined} if for each matrix $B\in    M_{m,n}$ with $B-A\in \mathfrak{m}^{k+1}\cdot   M_{m,n}$, we have $B\mathop\sim\limits^{G} A$, i.e. if $A$ is $G$-equivalent to every matrix which coincides with $A$ up to and including terms of order $k$. A matrix $A$ is called {\it finitely $G$--determined} if there exists a positive integer $k$ such that it is $G$ $k$--determined.
\vskip 5pt
Over the complex numbers, the classical criterion for finite determinacy says that a matrix $A\in M_{m,n}$ is finitely $G$-determined if and only if the tangent space at $A$ to the orbit $GA$ has finite codimension in $M_{m,n}$ (see \cite{BK16,Mat68,Wal81}). Over fields of arbitrary characteristic, finite $G$--determinacy was first studied for one power series,  i.e. $m=n=1$, as a key ingredient for classification of singularities (see \cite{BGM12,GK90}), and has been developed to matrices of power series recently in \cite{GP16,GP17a, Pha16}. It was shown in \cite{GP16, GP17b} that in positive characteristic the tangent space to the orbit $GA$  in general does not coincide with the image of the tangent map of the orbit map. For $A\in M_{m,n}$ instead of the tangent space we consider  the $R$-submodule of $M_{m,n}$
\[
\widetilde T_A(GA):=\langle E_{m, pq}\cdot A\rangle +\langle A\cdot E_{n, hl}\rangle+ \mathfrak{m}\cdot\left\langle\frac{\partial A}{\partial x_\nu}\right\rangle,
\]
which is the image of the tangent map of the orbit map $G\to GA$, and call it the {\textit {tangent image}} at $A$ to the orbit $GA$ in \cite{GP16}.  Here $\langle E_{m, pq}\cdot A\rangle$ is the $R$-submodule generated by  $E_{m, pq}\cdot A$, $p,q=1,\ldots,m$, with $E_{m,pq}$ the $(p,q)$-th canonical matrix of $Mat(m,m,R)$ ($1$ at place $(p,q)$ and $0$ elsewhere) and $\left\langle\frac{\partial A}{\partial x_\nu}\right\rangle$ is the $R$-submodule generated by the matrices $\frac{\partial A}{\partial x_\nu} = \left[\frac{\partial a_{ij}}{\partial x_\nu} ({\bf x})\right], \nu = 1, \ldots, s$. 
By replacing $\mathfrak{m}$ by $R$ in $\widetilde T_A(GA)$ we call the corresponding submodule $\widetilde T^e_A(GA)$ the {\textit {extended tangent image}} at $A$ to the orbit $GA$.
In arbitrary characteristic,  the following equivalent sufficient conditions for finite determinacy were obtained in  \cite[Proposition 4.2 and Theorem 4.3]{GP16}.
\begin{proposition}\label{sufficient}	
	{\rm 1.} Let $A\in \mathfrak{m}\cdot M_{m,n}$. Then $A$ is finitely $G$-determined if one of the following  equivalent statements holds:
	\begin{enumerate}
		\item[\rm (i)] \label{sufficient 1}  $\dim_K\big(\mathfrak{m}\cdot M_{m,n}/\widetilde T_A(GA)\big) =:d<\infty$.
		\item[\rm (ii)]  \label{sufficient 2} $\dim_K M_{m,n} \big{/}\widetilde T_A^e(GA) =:d_e<\infty.$
		\item[\rm (iii)]  \label{sufficient 3}  $\mathfrak{m}^k \subset I_{mn}\left(\Theta_{(G,A)}\right)$ for some positive integer $k$, where
		\[
		R^t\xrightarrow{\Theta_{(G,A)}} M_{m,n}  \to M_{m,n}/\widetilde T_A^e(GA)\to 0
		\]
		is a presentation of $M_{m,n}/\widetilde T_A^e(GA)$ and $I_{mn}(\Theta_{(G,A)})$ is the ideal of $mn\times mn$ minors of $\Theta_{(G,A)}$.
	\end{enumerate}
	Furthermore, if the condition (i) (resp. (ii) and (iii))  above holds  then $A$ is $G$ ($2c-\order(A)+2$)-determined, where $c= d$ (resp. $d_e$ and k) and $\order(A)$ is the minimum of the orders of entries of $A$. 
	
	\noindent{\rm 2.} If  $\characteristic(K)=0$ then the converse of 1. also holds.
\end{proposition}

The question whether in positive characteristic  the finite codimension of $\widetilde T_A(GA)$ is necessary for a matrix $A\in M_{m,n}$ to be finitely $G$-determined for arbitrary $m$ and $n$  remains open (see \cite[Conjecture 1.3]{GP17a}).
For the case of a one column matrix $A\in M_{m,1}$, the finite codimension of $\widetilde T_A(GA)$ is equivalent to finite $G$-determinacy of $A$ in arbitrary characteristic (see \cite{GP17a}). 
\vskip 5pt 	
The above question is answered positively for the case of $2\times 2$ matrices in this short note, where we  prove that the finite codimension of $\widetilde T_A(GA)$ is a necessary and sufficient criterion for a matrix $A\in M_{2,2}$ to be finitely $G$-determined. In order to do that we first prove that there exist finitely $G$-determined $2\times 2$ matrices $A$ of homogeneous polynomials of arbitrarily high order by showing that the ideal generated by the maximal minors of a  presentation matrix $\Theta_{(G,A)}$ of the $R$-module $M_{2,2}/\widetilde T^e_A(GA)$ is Artinian (Proposition \ref{existence}). 		
\vskip 5pt
The main result of the paper is the following theorem.
\begin{theorem}\label{main}
	Let $A\in \mathfrak{m}\cdot M_{2,2}$. Then the following are equivalent:
	\begin{enumerate}
		\item [\rm 1.] $A$ is finitely $G$-determined.
		\item [\rm 2.] $\dim_K M_{2,2}/\widetilde T_A(GA)<\infty.$  
	\end{enumerate}
\end{theorem}
\section{Proof of Theorem \ref{main}}\label{sect 2}
We show in Proposition \ref{existence} the important fact that there exist finitely $G$-determined matrices in $M_{2,2}$ of arbitrarily high order in arbitrary characteristic. Its proof  shows furthermore that the coefficients of the entries of such a matrix belong to a Zariski open subset of $K^{4s}$.
\begin{proposition}\label{existence}
	Let $\characteristic(K)=p\ge 0$ and ${B=\big[
		\begin{smallmatrix}
		f_{11}&f_{12}\\ 
		f_{21}&f_{22}
		\end{smallmatrix}
		\big]}$, where 
	\[
	{f_{ij}=c_{ij}^{(1)}x_1^{\,N} + \cdots + c_{ij}^{(s)}x_s^{\,N}}\in K[[{\bf x}]],
	\]
	$p\nmid N$ if $p>0$, and $c_{ij}^{(k)}$ are general elements in $K$. Then 
	$$\dim_KM_{2,2}/\widetilde T^e_B(GB)<\infty$$
	and $B$ is finitely $G$-determined.
\end{proposition}

\begin{proof}
	We first claim that the ideal of $4\times 4$ minors of the presentation matrix 
	\[
	M:=\Theta_{(G,B)}=
	\left[ {\begin{array}{*{20}c}
		{f_{11}} & f_{12} & 0 & 0 & f_{11} & f_{21} & 0 & 0 & \frac{\partial f_{11}}{\partial x_1} & \cdots & \frac{\partial f_{11}}{\partial x_s} \\
		f_{21} & f_{22} & 0 & 0 & 0 & 0 & f_{11} & f_{21} & \frac{\partial f_{21}}{\partial x_1} & \cdots & \frac{\partial f_{21}}{\partial x_s} \\
		0 & 0 & f_{11} & f_{12} & f_{12} & f_{22} & 0 & 0 & \frac{\partial f_{12}}{\partial x_1} & \cdots & \frac{\partial f_{12}}{\partial x_s} \\
		0 & 0 & f_{21} & f_{22} & 0 & 0 & f_{12} & f_{22} & \frac{\partial f_{22}}{\partial x_1} & \cdots & \frac{\partial f_{22}}{\partial x_s} 
		\end{array}} \right]
	\]
	in $K[{\bf x}]$ is Artinian.
	
	Indeed, let ${P=(a_1,\ldots,a_s)}$ be a point where all ${4 \times 4}$ minors of $M$ vanish. Denote by $M_{i_1i_2i_3i_4}$ the ${4 \times 4}$ minor of $M$ obtained from columns ${i_1,\ldots,i_4}$. Given the generality of the coefficients $c_{ij}^{(k)}$, we may assume that the determinant of $B$ does not vanish in any of the points ${(1,0,\ldots,0),\ldots,(0,\ldots,0,1)}$. We may also assume that all minors (of any size) of the matrix
	\[
	\begin{bmatrix}
	c_{11}^{(1)} & \cdots & c_{11}^{(s)} \\[1ex] 
	c_{21}^{(1)} & \cdots & c_{21}^{(s)} \\[1ex] 
	c_{12}^{(1)} & \cdots & c_{12}^{(s)} \\[1ex] 			c_{22}^{(1)} & \cdots & c_{22}^{(s)} 
	\end{bmatrix}
	\] 
	are non-zero. 
	
	First note that ${M_{1,2,3,4}=\det(B)^2}$, so the determinant of $B$ must vanish at $P$. Now observe that if ${s\ge4}$ then a ${4 \times 4}$ minor of $M$ taken from four of the last $s$ columns can be written as
	\[
	M_{i_1i_2i_3i_4}=
	\begin{vmatrix}
	c_{11}^{(i_1)} & c_{11}^{(i_2)} & c_{11}^{(i_3)} & c_{11}^{(i_4)} \\[1ex] 
	c_{21}^{(i_1)} & c_{21}^{(i_2)} & c_{21}^{(i_3)} & c_{21}^{(i_4)} \\[1ex] 
	c_{12}^{(i_1)} & c_{12}^{(i_2)} & c_{12}^{(i_3)} & c_{12}^{(i_4)} \\[1ex] 
	c_{22}^{(i_1)} & c_{22}^{(i_2)} & c_{22}^{(i_3)} & c_{22}^{(i_4)} 
	\end{vmatrix}
	\cdot N^4\ x_{i_1}^{\,N-1}\ x_{i_2}^{\,N-1}\ x_{i_3}^{\,N-1}\ x_{i_4}^{\,N-1}.
	\]
	Since the determinant is non-zero, at least one of each four coordinates of $P$ must be zero. In other words, $P$ has at most $3$ non-zero coordinates.  Without loss of generality, we may assume that ${a_4=\cdots=a_s=0}$. Suppose that $P$ is non-zero. Then at least two of $a_1$, $a_2$, and $a_3$ are non-zero, given our assumptions on the vanishing of $\det(B)$, and we may assume that ${a_1a_2\ne0}$. Note that for ${1\le i_1<i_2\le 8}$,
	\[
	M_{i_1,i_2,9,10}=F\cdot N^2x_1^{\,N-1}x_2^{\,N-1},
	\]
	where $F$ is a form in degree $2N$, involving only the $\binom{s+1}{2}$ monomials of type $x_i^{N}x_j^{N}$, with ${i\le j}$. Writing ${y_{ij}=x_i^{N}x_j^{N}}$, we can regard $F$ as a linear form on the variables $y_{ij}$. Now, since ${a_1a_2\ne0}$, we see that $F$ vanishes on $P$. Together with $\det(B)$, minors $M_{1,5,9,10}$, $M_{2,4,9,10}$, $M_{3,7,9,10}$, $M_{6,8,9,10}$, and $M_{3,6,9,10}$ form a system of $6$ linear equations on the variables $y_{ij}$. Since $P$ satisfies ${y_{ij}=0}$ for ${j>3}$, we can regard this as a system on the $6$ variables $y_{11}, y_{12}, y_{13}, y_{22}, y_{23}$, and $y_{33}$. Therefore if the system is independent, the only solution is zero. We can check that this is indeed the case by taking one parameter $a$, and assigning in the system of equations ${c_{11}^{(1)}=a}$, ${c_{12}^{(3)}=c_{21}^{(2)}=c_{22}^{(1)}=c_{22}^{(2)}=c_{22}^{(3)}=1}$, and ${c_{ij}^{(k)}=0}$ otherwise. Then the determinant of the system is $a^7+a^6$, which is non-zero.  This implies that $P=0$ for a general choice of $c_{ij}^{(k)}$, which finishes the proof of the claim.
	
	Applying now Proposition \ref{sufficient}, the statements follow.
\end{proof}

We need in addition the semi-continuity of the $K$-dimension of a $1$-parameter family of modules over a power series ring, which was obtained in \cite[Proposition 3.4]{GP17a}.

\begin{proposition}\label{semi-continuity}
	Let  $P=K[t][[{\bf x}]]$, 	where ${\bf x}=(x_1,\ldots,x_s)$, and $M$ a finitely generated $P$-module. For $t_0\in K$, set
	\begin{align*}
		M(t_0):&=M\mathop\otimes\limits_{K[t]}\ (K[t]/\langle t-t_0\rangle)
		\cong M/\langle t-t_0\rangle\cdot M.
	\end{align*}
	Then there is a nonempty open neighborhood $U$ of $\ 0$ in $\A_K^1$ such that for all $t_0\in U$, we have
	\[\dim_KM(t_0)\le\dim_KM(0).\]
\end{proposition}

\noindent{\textit{Proof of Theorem \ref{main}:}} By Proposition \ref{sufficient}, it suffices to prove the implication ($1.\Rightarrow 2.$). Assume that $A$ is $G$ $k$-determined. By finite determinacy we may assume that $A$ is a matrix of polynomials. Let $\characteristic(K)=p>0$ and $N\in \N$ such that $N>k$ and $p\nmid N$. Let $B\in M_{2,2}$
be a matrix as in Proposition \ref{existence}. Consider 
\[
B_t=B+tA\in Mat(2,2,K[t][{\bf x}])
\]
and the $K[t][[{\bf x}]]$-module
\[\widetilde T^e_{B_t}(GB_t)=\left\langle E_{ij}\cdot B_t, i,j=1,2\right\rangle+\left\langle  B_t\cdot E_{ij}, i,j=1,2\right\rangle+\left\langle \frac{\partial B_t}{\partial x_1},\ldots, \frac{\partial B_t}{\partial x_s} \right\rangle.
\]
Then by Proposition \ref{semi-continuity} there is a nonempty open subset $U\subset \A^1_K$ such that for all $t_0\in U$ we have
\[
\dim_K\bigl(M_{2,2}/\widetilde T^e_{B_{t_0}}(GB_{t_0})\bigr)\le \dim_KM_{2,2}/\widetilde T^e_{B}(GB)<\infty,
\]
where the second inequality follows from Proposition \ref{existence}. Let $t_0\in U$ and $t_0\ne 0$. Since $A\mathop\sim\limits^G t_0A \mathop\sim\limits^G  B_{t_0}$, we have
\[
\dim_K M_{2,2}/\widetilde T^e_{A}(GA)=\dim_K\bigl(M_{2,2}/\widetilde T^e_{B_{t_0}}(GB_{t_0})\bigr)<\infty,
\]
which is equivalent to the finiteness of the codimension of $\widetilde T_A(GA)$.\qed

\section*{Acknowledgements}
We would like to thank Professor Gert-Martin Greuel for  helpful suggestions and comments. Many thanks are also due to the anonymous referee for his/her suggestions. The first author was partially supported by the European Union's Erasmus+ programme. She would also like to thank the Vietnam Institute for Advanced Study in Mathematics (VIASM) for its support and  hospitality.
The second author was partially supported by CIMA -- Centro de Investiga\c{c}\~{a}o em Matem\'{a}tica e Aplica\c{c}\~{o}es, Universidade de \'{E}vora, project PEst-OE/MAT/UI0117/2014 (Funda\c{c}\~{a}o para a Ci\^{e}ncia e Tecnologia).

 Department of Mathematics and Statistics, Quy Nhon University, 170 An Duong Vuong, Quy Nhon, Vietnam

 Email address: phamthuyhuong@qnu.edu.vn

\vskip 7pt

Departamento de Matem\'{a}tica, Escola de Ci\^{e}ncias e Tecnologia, Centro de Inves\-ti\-ga\c{c}\~{a}o em Matem\'{a}tica e Aplica\c{c}\~{o}es, Instituto de Investiga\c{c}\~{a}o e Forma\c{c}\~{a}o Avan\c{c}ada, Universidade de \'{E}vora, Rua Rom\~{a}o Ramalho, 59, P--7000--671 \'{E}vora, Portugal

Email address: pmm@uevora.pt
\end{document}